\documentclass[reqno]{amsart}

\usepackage[english]{babel}
\usepackage{graphicx}
\usepackage{textcomp}
\usepackage{amsmath}
\newtheorem{proposition}{Proposition}[section]
\newtheorem{remark}{Remark}[section]
\newtheorem{corollary}{Corollary}[section]
\newtheorem{theorem}{Theorem}[section]
\begin{document}


\markboth{O.S.\,Rozanova}{Stochastic perturbations method for a
system of Riemann invariants}
\title[Stochastic perturbations method for a system of Riemann invariants]
{Stochastic perturbations method for a system of Riemann
invariants\thanks{Supported by  RFBR Project Nr. 12-01-00308 and by
the government grant of the Russian Federation for projects
implemented by leading scientists, Lomonosov Moscow State University
under the agreement No. 11.G34.31.0054}}


\author[O.S. Rozanova]{Olga S. Rozanova}
\address{Department of Mechanics and Mathematics, Moscow
State University, Moscow 119991 Russia} \email{{\tt
rozanova@mech.math.msu.su} (O.S.Rozanova)}

\begin{abstract}
Basing on our results \cite{AKR} on a representation of  solutions
to the Cauchy problem for multidimensional non-viscous Burgers
equation obtained by a method of stochastic perturbation of the
associated Langevin system, we deduce an explicit asymptotic formula
for  smooth solutions to the Cauchy problem for any genuinely
nonlinear hyperbolic system of equations written in the Riemann
invariants.
\end{abstract}


\date{\today}

\maketitle


\section{Introduction}

We consider a system of the following form:
\begin{equation}\label{RI}
\frac{\partial r_i}{\partial t} + f_i(r_1,\dots,r_n)\frac{\partial
r_i}{\partial x}=0,\quad i=1,\dots,n,
\end{equation}
written for the functions $r_i=r_i(t,x), \,t\in\mathbb \overline
R_+,\,x\in\mathbb R,$ subject to initial data
\begin{equation}\label{RI_id}
r(0,x)=r^0(x)=(r^0_1(x),\dots,r^0_n(x)).
\end{equation}

We assume that the vector-function $f(r)\in C^1({\mathbb R}^n)$ is
real-valued and
\begin{equation}\label{nongenerate}
{\rm det} \left(\frac{\partial f_i(r)}{\partial r_j}\right) \ne
0,\quad i,j=1,...,n.
\end{equation}

It is well known that any quasilinear hyperbolic system  $u_t+
A(u)u_x=0$, where $u=u(t,x)$ is a $n$-vector ${\mathbb R}^2\to
{\mathbb R}^n$, $A(u)$ is a $(n\times n)$ matrix with smooth
coefficients, can be written in the Riemann invariants if it
consists of two equations. For three equations the system can be
written in the Riemann invariants is and only if the left
eigenvectors $l_k$ of $A$ satisfies either ${\rm rot}\, l_k=0$ or
$l_k \cdot {\rm rot}\, l_k=0$ \cite{Stepanov}. Among physically
meaningful systems admitting the Riemann invariants one can name the
equations of isentropic gas dynamics, equations of shallow water,
equations of electrophoresis and equations of chromatography
\cite{Rogd}, \cite{Wis}, \cite{Dafermos}.

\section{Extended system and stochastic differential equation associated with the equations of characteristics
}\label{perturb} Basing on (\ref{RI}) we consider a system
\begin{equation}\label{RI_e}
\frac{\partial q_i}{\partial t} + f_k(q_1,\dots,q_n)\frac{\partial
q_i}{\partial x_k}=0,\quad k=1,\dots,n,
\end{equation}
written for functions $q_i(t,x), \,x\in \mathbb{R}^n,$ together with
initial data
\begin{equation}\label{RI_id_e}
q(0,x)=q^0(x)=(q^0_1(x),\dots,q^0_n(x)).
\end{equation}
Let us suppose that $x_i,\,i=1,...,n,$ are in its turn functions of
one variable $\bar x\in {\mathbb R}$. In this case the system
\eqref{RI_e} should be rewritten as
\begin{equation}\label{RI_e1}
\frac{\partial q_i}{\partial t} + f_k(q_1,\dots,q_n)\frac{\partial
q_i}{\partial x_k}\frac{d x_k(\bar x)}{d\bar x}=0,\quad
k,i=1,\dots,n,
\end{equation}
It can be readily checked that if we set
\begin{equation}\label{q_r}q_i(t,0,\dots,\underbrace{\bar x}_{i-\rm th \,place},\dots,0):=r_i(t,\bar x),
\end{equation}
 such
that $\frac{\partial q_i}{\partial x_m}=0,\,i\ne m,$ and
\begin{equation}\label{q_r_0} q_i^0(0,\dots,\underbrace{\bar
x}_{i-\rm th \,place},\dots,0):=r_i^0(\bar x),\end{equation} $
i,m=1,\dots,n,$ then the vector-function $r(t,\bar x)$ solve the
problem \eqref{RI}, \eqref{RI_id}.

We associate with (\ref{RI_e})  the following system of stochastic
differential equations:

\begin{equation}\label{R_Rim_inv_SDE}
\begin{array}{cc}
dX_i(t)=f_i(Q_1(t),..., Q_n(t))dt+\sigma_1 d(W_i^1)_t,\\ \\
dQ_i(t)=\sigma_2\, d(W_i^2)_t,\quad i=1,...,n,\\  \\X_i(0)=x_i,\quad
Q_i(0)=q_i,\quad t>0,\end{array}
\end{equation}
where $X(t)$ and $Q(t)$ - are random values with given initial data,
$(X(t),Q(t))$ takes values in the phase space
$\mathbb{R}^n\times\mathbb{R}^n,$ $\sigma_1$ and $\sigma_2$ - are
positive constants, $|\sigma|\ne 0$ ($\sigma=(\sigma_1,\sigma_2)$),
è $(W^j)_t=(W^j_1,\dots,W^j_n)_t$, $j=1,2$, are independent
Brawnnian motions.

Let $P(t,d q_1,...,d q_n,dx_1,...,dx_n),\, t\in \mathbb R_+,\,x_i\in
\mathbb R,\,q_i\in \mathbb R,\,i=1,..,n,$  be  the joint probability
distribution of the random variables $(Q,X), $ subject to the
initial data
\begin{equation}
\label{P0} P_0(dr,dx)= \prod\limits_{i=1}^n \delta_q(q_i^0(x_i))
\rho_i^0(x_i)\,dx=\delta_q(q^0(s))\,\rho^0(s),
\end{equation}
 where $\rho_i^0(x_i)$ are  bounded nonnegative
functions from $C ({\mathbb R})$ and $dx$ is Lebesgue measure on
${\mathbb R}^n$,
$\delta_q$ is Dirac measure concentrated on $q$.

We look at  $P=P(t,dq,dx)$ as a generalized function (distribution)
with respect to the variable $q$. It satisfies the Fokker-Planck
equation
\begin{equation}
\label{Fok-Plank} P_t+\sum\limits_{i=1}^n\,f_i(q)
P_{x_i}=\frac{\sigma_1^2}{2}\sum\limits_{i=1}^n\, P_{x_i
\,x_i}+\frac{\sigma_2^2}{2}\sum\limits_{i=1}^n\, P_{q_i \,q_i}
\end{equation}
with initial data (\ref{P0}).


There is a standard procedure for finding the fundamental solution
for \eqref{Fok-Plank} (see, e.g. \cite{Friedman}). This procedure
consists in a  reduction of the equation to a Fredholm integral
equation, the solution of which can be found in the form of series.
We are going to show that one can also find an explicit solution to
the Cauchy problem \eqref{Fok-Plank}, \eqref{P0}.

Let us introduce, still in the general case, the functions, for
$t\in \mathbb R_+,\,x\in \mathbb R^n$, depending on
$\sigma=(\sigma_1, \sigma_2)$:
\begin{equation}
\label{rho_i_s} \rho_{i,\sigma}(t,x_i)=\int\limits_{{\mathbb
R}^{2n-1}}\,P(t,x,dr)\,d\breve{x_i},
\end{equation}
\begin{equation}
\label{rho_s} \rho_{\sigma}(t,x)=\int\limits_{{\mathbb
R}^{n}}\,P(t,x,dq),
\end{equation}
\begin{equation}
\label{q_s}q_{i,\sigma}(t,x)=\frac{\int\limits_{{\mathbb
R}^{n}}\,q_i\,P(t,x,dq)}{\int\limits_{{\mathbb R}^{n}}\,P(t,x,dq)},
\end{equation}
\begin{equation}
\label{f_s} f_{i,,\sigma}(t,x)=\frac{\int\limits_{{\mathbb
R}^{n}}\,f_i(q)\,P(t,x,dq)}{\int\limits_{{\mathbb
R}^{n}}\,P(t,x,dq)},
\end{equation}
where
$$
d\breve{x_i}\,=\,dx_1\,..\,dx_{i-1}\,dx_{i+1}\,...\,dx_n,\quad
dr\,=\,dr_1...dr_n.
$$

We can consider these values if the integrals  exist in the Lebesgue
sense.

It will readily be observed that $q_{i,\sigma}(0,x)=q_i^0(x)$ and
$f_{i,\sigma}(0,x)=f(q_i^0(x))$.

We denote
$$\bar \rho_i (t,\bar x)=\lim\limits_{\sigma\to 0} \rho_{i,\sigma}(t,\bar x),
\quad \bar \rho (t, x)=\lim\limits_{\sigma\to 0}
\rho_{\sigma}(t,x),$$$$ \bar
q_{i,\sigma}(t,x)=\lim\limits_{\sigma\to 0} q_{i,\sigma}(t,x),\quad
\bar f_i (t, x)=\lim\limits_{\sigma\to 0} f_{i,\sigma}(t, x),$$
provided these limits exist,  $\sigma = \sqrt{\sigma_1^2+\sigma_2^2}
$.

\section{Explicit probability density function}

The equation \eqref{Fok-Plank} can be solved explicitly. Moreover,
for the sake of simplicity we set $\sigma_2=0$ and denote
$\sigma_1=\sigma$.

\begin{proposition}\label{Prop} The problem
\eqref{Fok-Plank}, \eqref{P0} has the following solution:
\begin{eqnarray}\label{s_plotn}
 P(t,x,dq) =\hspace{10cm}     \\
 \dfrac1{(\sqrt{2\pi
t}\sigma)^n}\int\limits_{\mathbb{R}^n}\,\delta_r(r^0(s))\,\rho^0(s)\,\exp\left(
-\frac{\sum\limits_{i=1}^n (f_i(r^0_1(s_1),...,r^0_n(s_n))\,t
\,+\,(s_i-x_i))^2}{2\sigma^2 t}\right) \,ds, \nonumber \quad
\end{eqnarray}
for $\quad t\ge 0, \,x\in \mathbb{R}^n,$ therefore
\begin{eqnarray}\label{s_plotn_1}
 \int\limits_{{\mathbb R}^n} \, \phi(q)\,P(t,x,dq) = \hspace{10cm}     \\=\dfrac1{(\sqrt{2\pi
t}\sigma)^n}\int\limits_{\mathbb{R}^n}\,\phi(r^0(s))\,\rho^0(s)\,
\,\exp\left( -\frac{\sum\limits_{i=1}^n
(f_i(r^0_1(s_1),...,r^0_n(s_n))\,t \,+\,(s_i-x_i))^2}{2\sigma^2
t}\right) \,ds,\nonumber
\end{eqnarray}
for all $\phi(r)\in C_0({\mathbb R}^n)$.

\end{proposition}

\begin{proof} Let us apply the Fourier transform to $P(t,x,dq)$ in
(\ref{Fok-Plank}), (\ref{P0}) with respect to the variable $x$ and
obtain the Cauchy problem for the Fourier transform
$\tilde{P}=\tilde{P}(t,\lambda,dq)$ of $P(t,x,dq)$:
\begin{equation}
\label{preobr_Fok-Plank} \dfrac{\partial \tilde{P}}{\partial
t}=-(\dfrac12\sigma^2|\lambda|^2+i(\lambda, f(q)))\tilde{P},
\end{equation}
\begin{equation}
\label{preobr_P0}\tilde{P}(0,\lambda,dq)=\int\limits_{\mathbb{R}^n}e^{-i(\lambda,s)}\delta_q(q^0(s))\rho^0(s)ds,
\qquad \lambda \in {\mathbb R}^n.
\end{equation}
Equation (\ref{preobr_Fok-Plank}) can easily be  integrated and we
obtain the solution given by the following formula:
\begin{equation}
\label{preobr_P}\tilde{P}(t,\lambda,dq)=\tilde{P}(0,\lambda,dq)e^{-\frac12\sigma^2|\lambda|^2t
+i  (\lambda,f(q))\,t}.
\end{equation}
The inverse Fourier transform (in the distributional sense) allows
to find the density function $P(t,x,dq),\,t>0$:
$$P(t,x,dq)=\dfrac1{(2\pi)^{n}}
\int\limits_{\mathbb{R}^n}e^{i(\lambda,x)}\tilde{P}(t,\lambda,dq)\,d\lambda
=$$
$$=\dfrac1{(2\pi)^{2n}}\int\limits_{\mathbb{R}^n}e^{i(\lambda,x)}
\left(\int\limits_{\mathbb{R}^n}e^{-i(\lambda,s)}e^{-i
(\lambda,f(q)) \, t}\,\delta_q(q^0(s))\,\rho^0(s)ds\right)
\,e^{-\frac12\sigma^2|\lambda|^2t}d\lambda=$$
$$=\dfrac1{(2\pi)^{n}}\int\limits_{\mathbb{R}^n}\delta_q(q^0(s))\,\rho^0(s)
\int\limits_{\mathbb{R}^n}e^{-\frac12\sigma^2t\left(\lambda-\frac{i|f(q)t+s-x|}{\sigma^2t}\right)^2-\frac{|f(q)t
+s-x|^2}{2\sigma^2t}}d\lambda ds=$$
$$
=\dfrac1{(\sqrt{2\pi
t}\sigma)^n}\int\limits_{\mathbb{R}^n}\,\delta_q(q^0(s))\,\rho^0(s)\,e^{-\frac{|
f(q^0(s))t+s-x|^2}{2\sigma^2t}}ds,\quad t\ge
0,\,x\in \mathbb{R}^n. $$
The third equality is satisfied by Fubini's theorem, which can be
applied by the absolute integrability and the bound on the function
involved. Thus, the proposition is proved.
\end{proof}

\begin{remark} In the general case $\sigma_2\ne 0$ an analogous
formula can be obtained in a similar way (see \cite{KR} in this
context).
\end{remark}

\begin{corollary} \label{corol} The functions ${\rho}_\sigma$, ${q}_\sigma$ and
$f_\sigma$ defined in \eqref{rho_s} -- \eqref{f_s}  can be
represented by the following formulae:
\begin{equation}
\label{rho_s_rep} {\rho}_\sigma(t,x)=\dfrac1{(\sqrt{2\pi
t}\sigma)^n}
{\int\limits_{\mathbb{R}^n}\rho^0(s)\,e^{-\frac{\sum\limits_{i=1}^{n}
|{f_i(r^0(s))t+s_i-x_i|^2}}{2\sigma^2t}} ds},
\end{equation}
\begin{equation}
\label{q_s_rep} {q}_{\sigma,i}(t,x)=
\dfrac{\int\limits_{\mathbb{R}^n}r^0(s)\rho^0(s)\,e^{-\frac{\sum\limits_{i=1}^{n}|f_i(r^0(s))t
+ s_i-x_i|^2}{2\sigma^2t}} ds}
{\int\limits_{\mathbb{R}^n}\rho^0(s)\,e^{-\frac{\sum\limits_{i=1}^{n}
|f_i(r^0(s))t+s_i-x_i|^2}{2\sigma^2t}} ds},\quad i=1,...,n,
\end{equation}
\begin{equation}
\label{f_s_rep} {f}_{\sigma,i}(t,x)=
\dfrac{\int\limits_{\mathbb{R}^n}f(r^0(s))\rho^0(s)\,e^{-\frac{\sum\limits_{i=1}^{n}|f_i(r^0(s))t
+ s_i-x_i|^2}{2\sigma^2t}} ds}
{\int\limits_{\mathbb{R}^n}\rho^0(s)\,e^{-\frac{\sum\limits_{i=1}^{n}
|f_i(r^0(s))t+s_i-x_i|^2}{2\sigma^2t}} ds},\quad i=1,...,n.
\end{equation}

\end{corollary}

\begin{proof}  The result is obtained by substitution of
$P(t,x,dq)$ as given by \eqref{s_plotn} in \eqref{rho_s},
\eqref{q_s}, \eqref{f_s}.\end{proof}

\section{Representation of smooth solution to \eqref{RI}, \eqref{RI_id}}\label{asympt_formula}

Now we are going to prove that if $\rho_i$ and  $\bar q_i$ are
continuous, then
 $\bar r_i(\bar x)=\bar q_i(t,x)|_{\{x_j=0, \,x_i=\bar x\}},$ $j\ne i,$ tend to the solution
$r_i(t,\bar x)$ of  the problem  \eqref{RI}, \eqref{RI_id}  as
$\sigma \to 0$.

Namely, the following theorem holds.

\begin{theorem}\label{T1} Let $r(t,\bar x)$ be a solution to the Cauchy problem
\eqref{RI}, \eqref{RI_id},
 $r^0\in
C^1_b({\mathbb R})$ and $t_*(r^0)$ be the supremum of $t$ such that
this solution  is smooth. Then for $t\in [0, t_*(r^0)),$
$$r_i(t,\bar x)=\bar q_i(t, x)|_{\{x_j=0, \,x_i=\bar x\}}=\lim\limits_{\sigma\to 0} q_{\sigma,i}(t,x)|_{\{x_j=0, \,x_i=\bar x\}},\,j\ne i,$$
where $q_\sigma(t,x)$ is given by (\ref{q_s_rep}) and the limit
exists pointwise.
\end{theorem}

\begin{proof} The easiest way to prove the theorem is reducing of
\eqref{RI_e} to the multidimensional non-viscous Burgers equation
and using the representation from \cite{AKR}. Namely, \eqref{RI_e}
has the form
\begin{equation}\label{ext_Burgers}
\partial_t q+(f(q),\nabla)q=0,
\end{equation}
where $q(t,x)=(q_1,...,q_n)$ is a vector-function
$\mathbb{R}^{n+1}\rightarrow\mathbb{R}^n,$ $f(q)$ is a
non-degenerate differential  mapping from ${\mathbb R}^n\,$ to
$\,{\mathbb R}^n, $ such that its Jacobian satisfies the condition
${\rm det} \frac{\partial f_i(q)}{\partial q_j} \ne 0,$
$i,j=1,...,n,$ due to \eqref{nongenerate}.

We  multiply (\ref{ext_Burgers}) by $\nabla_q f_i(q),\,i=1,...,n,$
to get
\begin{equation}\label{ext_Burgers1}
\partial_t f(q)+(f(q),\nabla)f(q)=0.
\end{equation}
Thus, we can introduce a new vectorial variable $u=f(q)$ to reduce
the Cauchy problem for (\ref{ext_Burgers1}) to
\begin{equation}
\label{equ_Burg} \partial_t u+(u,\nabla)u=0,\,t>0,\qquad
u(x,0)=u_0(x)\in C^1(\mathbb{R}^n)\cap C_b(\mathbb{R}^n).
\end{equation}
 with $u_0(x)=f(q_0(x)).$
As follows from \cite{AKR}, the solution to the nonviscous Burgers
equation \eqref{equ_Burg} before the moment $t_*$ of a singularity
formation can be obtained as a pointwise limit as $\sigma\to 0$ of
\begin{equation}
\label{sol_u_sdu}
{u}_\sigma(t,x)=\dfrac{\int\limits_{\mathbb{R}^n}u_0(s)\rho^0(s)e^{-\frac{|u_0(s)t+s-x|^2}{2\sigma^2t}}ds}
{\int\limits_{\mathbb{R}^n}\rho^0(s)e^{-\frac{|u_0(s)t+s-x|^2}{2\sigma^2t}}ds}.
\end{equation}

Therefore we  find the representation of the solution to the
stochastically perturbed along the characteristics  equation
(\ref{ext_Burgers1}) using the formula (\ref{sol_u_sdu}) with
$f(q_0(x))$ instead of $u_0(x).$

Now we can go back to the vector-function $q(t,x)$. As follows from
Proposition~\ref{Prop},   $\bar q_i(t, x)=\lim\limits_{\sigma\to 0}
q_{\sigma,i}(t,x).$  At last, we use the redesignation \eqref{q_r}
to obtain the statement of  theorem \ref{T1}.
\end{proof}

Further, we can find the maximal time $t_*(r^0)$ of existence of the
smooth solution to the problem \eqref{RI}, \eqref{RI_id}.

\begin{theorem}\label{T2}
If at least one derivative
\begin{equation}\label{cond_t_*}
\frac{\partial{f_i(q)}}{\partial q_i}\,\frac{d r_i^0(x)}{d x}, \quad
i=1,\dots,n,
\end{equation}
is negative, then the time $t_*(r^0)$ of existence of the smooth
solution to \eqref{RI}, \eqref{RI_id} is finite and
\begin{equation}\label{t_*}
t_*(r^*)=\min\limits_{i}\left\{-\left({\frac{\partial{f_i(q)}}{\partial
q_i}\,\frac{d r_i^0(x)}{d x}}\right)^{-1}\right\}.
\end{equation}
Otherwise, $t_*(r^0)=\infty.$
\end{theorem}

\begin{proof}  Let
us  come back to the problem \eqref{equ_Burg} and denote by
$J(u_0(x))$ the Jacobian matrix of the map $\,x\mapsto u_0(x)$.  As
it was shown in \cite{protter}\,(Theorem 1), if $J(u_0(x))$ has at
least one eigenvalue which is negative for a certain point
$x\in{\mathbb R}^n,$  then the classical solution to
(\ref{equ_Burg}) fails to exist beyond a positive time $t_*(u_0).$
Otherwise, $t_*(u_0)=\infty.$ The matrix $C(t,x) = (I+t J(u_0(x))),$
where $\,I\,$ is the identity matrix, fails to be invertible for
$t=t_*(u_0).$ Thus, due to representation \eqref{ext_Burgers1}, if
$J(f(q^0(x)))$ has at least one eigenvalue which is negative for a
certain point $x\in{\mathbb R}^n,$  then the classical solution to
(\ref{RI_e}) fails to exist beyond a positive time $t_*(q^0).$ For
the problem \eqref{RI}, \eqref{RI_id} this means that as least one
derivative \eqref{cond_t_*} is negative. The value \eqref{t_*} can
be found from the condition of invertibility of the matrix $(I+t
J(f(q_0(x)))$.
\end{proof}

\begin{remark} It easy to see that Theorem \ref{T2} can be
considered  as a version of Theorem 7.8.2 \cite{Dafermos}, with a
specification of the blow up time, obtained by a different method.
\end{remark}

\section{Balance laws, associated with systems written in the Riemann
invariants} Now we get one more corollary of results of \cite{AKR}.
Let us denote $$g_{i,\sigma}(t,\bar x)=f_{i,\sigma}(t, x)|_{\{x_j=0,
\,x_i=\bar x\}},\quad \bar g_i (t,\bar x)=\lim\limits_{\sigma\to 0}
g_{i,\sigma}(t,\bar x).$$


\begin{theorem}\label{T3} The functions $\rho_{i, \sigma}$ and
$g_{i,\sigma}$  satisfy  the following system of $\,2n\,$ equations:
\begin{equation}
\label{sist_rho}\dfrac{\partial\rho_{i,\sigma}}{\partial
t}\,+\,\partial_x (\rho_{i,\sigma}
g_{i,\sigma})\,=\,\dfrac12\sigma^2\dfrac{\partial^2\rho_{i,\sigma}}{\partial
x^2},
\end{equation}
\begin{equation}
\label{sist_f}\dfrac{\partial(\rho_{\sigma,i}{g}_{\sigma,i})}{\partial
t}\,+\,
\partial_x(\rho_{\sigma,i}\,{g}^2_{\sigma,i})\,=
\end{equation}
\begin{equation}\nonumber
\dfrac12\sigma^2\dfrac{\partial^2(\rho_{\sigma,i}
g_{\sigma,i})}{\partial
x^2}\,-\,\int\limits_{\mathbb{R}^{2n-1}}(g_{\sigma,i}\,-\,\bar
g(r)_i)\,\big((g_{\sigma}\,-\,{\bar g(r)}),\nabla_{ x}
P(t,x,dq)\big)d \breve{x}.
\end{equation}
For $t\in (0, t_*(r^0))$ its limit functions $\bar\rho_{i}$ and
$\bar g_{i}$ satisfy the system of $\,2 n\,$ conservation laws:
\begin{equation}
\label{sist_rho_i_l}\dfrac{\partial\bar\rho_{i}}{\partial
t}\,+\,\partial_x (\bar\rho_{i} \bar g_{i})\,=\,0,
\end{equation}
\begin{equation}
\label{sist_f_i_l}\dfrac{\partial(\bar\rho_{i}{\bar
g}_{i})}{\partial t}\,+\,
\partial_x(\bar\rho_{i}\,{\bar g}^2_{i})\,=\,0,
\end{equation}
$ i=1,..,n,\,t\ge 0.$
\end{theorem}

\begin{proof} The statement follows from Theorem 2.1 \cite{AKR}.
Namely, the theorem implies that the scalar function
$\rho_\sigma(t,x)$ and the vector-function $f_\sigma (t,x)$ solve
the following system:
\begin{equation}
\label{sist_rho_q}\dfrac{\partial\rho_\sigma}{\partial t}\,+\,{\rm
div}_x (\rho_\sigma
f_\sigma)\,=\,\dfrac12\sigma^2\sum\limits_{k=1}^{n}\dfrac{\partial^2\rho_\sigma}{\partial
x_k^2},
\end{equation}
\begin{equation}
\label{sist_f_q}\dfrac{\partial(\rho_\sigma f_{i,\sigma})}{\partial
t}\,+\, \nabla(\rho_\sigma\,f_{i,\sigma}\,f_\sigma)\,=
\end{equation}
\begin{equation}
\,\dfrac12\sigma^2\sum\limits_{k=1}^{n}\dfrac{\partial^2(\rho_\sigma
f_{\sigma,i})}{\partial
x_k^2}-\,\int\limits_{\mathbb{R}^n}(f_{i,\sigma}\,-\, \bar
f_i)\,\big((f_\sigma\,-\,{\bar f}),\nabla_x P(t,x,dq)\big),\nonumber
\end{equation}
with $ i=1,..,n,\,t\ge 0,$ and the integral term vanishes as
$\sigma\to 0$. To obtain the statement of Theorem \ref{T3}, it is
sufficient to set $x_j=0, \,x_i=\bar x$ and $\rho_j(x_j)=1,$ $j\ne
i,$ for every fixed $i$.
\end{proof}

\begin{remark} System \eqref{sist_rho_i_l}, \eqref{sist_f_i_l}
constitutes $n$ systems of so called "pressureless" gas dynamics,
the simplest model introduced to describe the formation of large
structures in the Universe, see, e.g. \cite{Shand}.

\end{remark}

\begin{remark} The method of stochastic perturbations allows to study solutions to quasilinear systems
written in Riemann invariants at the moment of the singularity
formation and the shock waves evolution as well (see in this context
\cite{AR}, \cite{AKR}, \cite{KR} for simpler cases). In particular,
it is possible to prove that after the moment $t_*(r^0)$ of
singularity formation in the solution to the problem \eqref{RI},
\eqref{RI_id} the limit system for $\rho_{i, \sigma}$ and
$g_{i,\sigma}$ differs from \eqref{sist_rho_i_l}, \eqref{sist_f_i_l}
and contains an additional integral term in the group of equations
\eqref{sist_f_i_l}. This term does not vanish as $\sigma \to 0$ and
can be considered as gradient of a specific pressure term.
\end{remark}

\end{document}